\documentclass{amsart}
\usepackage[numbers]{natbib}
\usepackage{graphicx,amssymb}

\newtheorem{thm}{Theorem}[section]

\theoremstyle{definition}
\newtheorem{definition}{Definition}
\theoremstyle{remark}
\newtheorem{remark}{Remark}
\numberwithin{equation}{section}

\begin{document}

\title{More About Solutions Of Quadratic Equations}
\author{{\tiny\bf  The  Accompanying Variables Method}\\Wolf-Dieter Richter}%
\address{Institute of Mathematics, University of Rostock,  Germany.
\newline\hspace*{.1cm}\quad ORCID 0000/0002/3610-7219}%
\email{wolf-dieter.richter@uni-rostock.de}%

\subjclass{12D10, 26C10}%
\keywords{zeros of quadratic equations; vector solutions; circular solutions;
generalized circular solutions; hyperbolic solutions; circular accompanying variables; hyperbolically accompanying variables}%

\begin{abstract}
For polynomials of degree two which have no zeros, the method of accompanying variables is developed and zeros of associated vector polynomials are determined. Our flexible method uses a wide variety of possible vector-valued vector squaring methods and can be adopted to a wide variety of application situations. Known solutions are made much more precise by replacing the imaginary component and supplemented by introducing a whole class of new vector solutions. Circular, generalized circular and hyperbolic solutions are considered. Anyone who follows the approach of this work and considers equations of third or higher degree will come across further conclusions for the imaginary numbers used there.

\end{abstract}
\maketitle

\section{Introduction}
In his dissertation \cite{Gauss1799}, C.F. Gauss proved that "Every
Algebraic Rational Integral Function
In One Variable can be Resolved into
Real Factors of the First or the Second Degree". It is well known when a real polynomial of degree two has no zeros. For such situations, the method of accompanying variables is developed in the present work.

In this approach, two-dimensional polynomials are assigned to a univariate polynomial in an infinite number of possible ways and their vector zeros are described. The choice of accompanying variables can be made depending on an application problem at hand.

The use of so-called imaginary and thus formed complex numbers is avoided. However, there is a close connection between the latter and pairs of accompanying variables which have their values on Euclidean circles, as occurs, for example, with normalized pairs of sine and cosine signals in electrical engineering.

When signals in time  follow generalized or modified sine and cosine curves as in medicine and various areas of technique, the introduction of accompanying variables, which take their values at fixe time on generalized or modified circles, respectively, becomes interesting. We assume that the generalized circles considered enclose discs which are star-shaped sets and are generated by positively homogeneous functionals.

These functionals also serve to define a variety of vector-valued vector products.
All these vector multiplications allow specific definitions of squaring in $R^2$. In this way, a large class of vector valued quadratic polynomials will be introduced whose zeros are subsequently studied.

 For a long time already, one might have the impression that the theory of the fundamental theorem of algebra can be considered to be largely rounded off. While the present work approaches the sub-topic of quadratic equations from a new perspective, the search for effective proofs of the theorem itself has also not yet come to an end, as one can find in \cite{NoMa}. A link to the English translation of \cite{Gauss1799} can also be found there.

\section{Ordinary observations}
We consider the real-valued function with real coefficients
\[f(x)=x^2+px+q,\,x\in R\]
which is strictly positive if
\begin{equation}\label{pos1}
q-\frac{p^2}{4}>0.
\end{equation}
In this case, the equation
\begin{equation}\label{01}
f(x)=0
\end{equation}
has no solution. In case
\begin{equation}\label{nonneg}
\frac{p^2}{4}-q\geq 0
\end{equation}
equation (\ref{01}) has the well-known solutions
\begin{equation}\label{Lsg1}
x_{1/2}=-\frac{p}{2} \overset{+}{-}\sqrt{\frac{p^2}{4}-q}
\end{equation}
and for $x\in (-\infty,x_2)\cup(x_1,\infty)$ holds $f(x)>0$ and for $x\in(x_1,x_2)$ inequality $f(x)<0$ holds true.

The following section is mainly devoted to the case where condition (\ref{pos1}) is fulfilled.

\section{Circular accompanying variables}
We now assume that there is a real variable $y$ accompanying the real variable $x$, possibly hidden, and consider the vector-valued function with real coefficients and the vector $1\hspace{-.13cm}{1}=\left(
                                                       \begin{array}{c}
                                                         1 \\
                                                         0 \\
                                                       \end{array}
                                                     \right)\in R^2,
$
\[f_{\circledast}(z)=z^{\circledast 2}\oplus p\cdot z\oplus q\cdot 1\hspace{-.13cm}{1},\;z=\left(
                                                   \begin{array}{c}
                                                     x \\
                                                     y \\
                                                   \end{array}
                                                 \right)\in R^2\]
where $p\cdot z=\left(
                  \begin{array}{c}
                    px \\
                    py \\
                  \end{array}
                \right)
$ means multiplication of vector $z$ by scalar $p\in R,$ $\left(
                                                            \begin{array}{c}
                                                              x_1 \\
                                                              y_1 \\
                                                            \end{array}
                                                          \right)
\oplus\left(
   \begin{array}{c}
     x_2 \\
     y_2 \\
   \end{array}
 \right)=\left(
           \begin{array}{c}
             x_1+x_2 \\
             y_1+y_2 \\
           \end{array}
         \right)
$ means common vector addition and $z^{\circledast  2}=z \circledast z$ denotes circular squaring of $z\in R^2$. Here, the circular vector-valued vector multiplication is defined according to \cite{Richter2020} by
\begin{equation}\label{Mult1}\left(
    \begin{array}{c}
      x_1 \\
      y_1 \\
    \end{array}
  \right)\circledast \left(
           \begin{array}{c}
             x_2 \\
             y_2 \\
           \end{array}
         \right)=\left(
                  \begin{array}{c}
                    x_1x_2-y_1y_2 \\
                    x_1y_2+x_2y_1 \\
                  \end{array}
                \right).
\end{equation}
Obviously,
\[1\hspace{-.13cm}{1}\circledast z=z, z\in R^2.\]
Let $C(r)=\{\left(
              \begin{array}{c}
                x \\
                y \\
              \end{array}
            \right)\in R^2:x^2+y^2=r^2\}, r>0$
            denote the Euclidean circle of radius $r$ and let $z_2$ be an arbitrary element of $C(1)$. For $z_1\in C(r_1)$ then holds $z_1\circledast z_2\in C(r_1)$, that is $z_2$ moves $z_1$ on the Euclidean circle $C(r_1)$. Such a movement is  complemented by a change from the circle $C(r_1)$ to the circle $C(r_1r_2)$ if the element from $C(1)$ is additionally multiplied by $r_2>0$ so that  $z_2\in C(r_2), r_2>0.$ In particular, with $z=\left(
                                  \begin{array}{c}
                                    x \\
                                    y \\
                                  \end{array}
                                \right),
            $
\begin{equation}\label{Quad1}
z^{\circledast  2}=z \circledast z=\left(
              \begin{array}{c}
                x^2-y^2 \\
                2xy \\
              \end{array}
            \right).
\end{equation}
The operations $\oplus$ and $\circledast$ are distributive.
We put $\lambda\cdot 1\hspace{-.13cm}{1}=
\lambda 1\hspace{-.13cm}{1} \text{ and }
(-1)\cdot 1\hspace{-.13cm}{1}=-1\hspace{-.13cm}{1},$
for simplicity. Obviously,
\begin{equation}\label{I1}
I\hspace{-.12cm}{I}^{\circledast  2}=-1\hspace{-.13cm}{1}
\end{equation}
where
\begin{equation}\label{I}
I\hspace{-.12cm}{I}
=\left(
    \begin{array}{c}
      0 \\
      1 \\
    \end{array}
  \right)\in R^2
\end{equation} also has the property of causing an orthogonal rotation,
\[\left(
    \begin{array}{c}
      x \\
      y \\
    \end{array}
  \right)\circledast I\hspace{-.12cm}{I}=\left(
    \begin{array}{c}
      -y \\
      x \\
    \end{array}
  \right).
\]
We now consider with $\mathfrak o=\left(
                                    \begin{array}{c}
                                      0 \\
                                      0 \\
                                    \end{array}
                                  \right)
$ the vector equation
\begin{equation}\label{02}
f_{\circledast}(z)=\mathfrak o.
\end{equation}
\begin{thm}
\label{Th1}
\begin{enumerate}
\item[(a)] If condition (\ref{pos1}) is fulfilled then the solutions of equation (\ref{02}) are
\begin{equation}\label{Lsg2}
z_{1/2}=-\frac{p}{2}1\hspace{-.13cm}{1}\overset{+}-\sqrt{q-\frac{p^2}{4}}I\hspace{-.12cm}{I}.
\end{equation}
\item[(b)] If condition (\ref{nonneg}) is fulfilled then with $x_{1/2}$ from (\ref{Lsg1}) the solutions of (\ref{02}) are
\begin{equation}\label{Lsg3}
z_{1/2}=\left(
          \begin{array}{c}
            x_{1/2} \\
            0 \\
          \end{array}
        \right).
\end{equation}
\end{enumerate}

\end{thm}

\begin{proof}We first write the vector equation (\ref{02}) as the system of univariate equations
\begin{equation}\label{A}
x^2-y^2+px+q=0
\end{equation}
and
\begin{equation}\label{B}
2xy+py=0.
\end{equation}
We now study the case $y=0$: Equation (\ref{B}) is trivially satisfied and equation (\ref{A}) agrees with equation (\ref{01}) which under assumption (\ref{nonneg}) has the solutions $x_{1/2}$ out of (\ref{Lsg1}), but under assumption (\ref{pos1}) has none. This proves (b).

We now move on to the case $y\neq 0$: Equation (\ref{B}) now means $x=-\frac{p}{2}$ and equation (\ref{A}) reduces to $y^2=q-\frac{p^2}{4}$. Under assumption (\ref{nonneg}), equation (\ref{A}) has no solution, but under (\ref{pos1}) holds $y=\overset{+}-\sqrt{q-\frac{p^2}{4}}$, so equation (\ref{02}) has the solutions out of (\ref{Lsg2}).

In summary, equation (\ref{02}) under assumption (\ref{nonneg}) has the solutions (\ref{Lsg3}) and under assumption (\ref{pos1}) the solutions (\ref{Lsg2})
\end{proof}
\begin{remark}
 It is essential for the understanding of this work that one must not "identify" the vector in (\ref{Lsg3}) with the scalar in (\ref{Lsg1}). Also, the product $\circledast$, borrowed from the current theory of complex numbers, is consistently used here only for vectors from $R^2$.
\end{remark}
In the following theorem we give a geometric description of the solution of equation (\ref{02}) that is equivalent to the anlytical description in Theorem \ref{Th1}.
\begin{thm}\label{Th2}
If condition (\ref{pos1}) is met, the solutions of the vector equation (\ref{02}) can be represented geometrically after the polar coordinates transformation $x=r  \cos\varphi, y=r  \sin\varphi$ as follows:
\begin{equation}\label{Los}
r=\sqrt{q} \quad\text{ and }\quad \cos\varphi=-\frac{p}{2\sqrt{q}},\; \sin\varphi=\overset{+}-\sqrt{1-\frac{p^2}{4{q}}}.
\end{equation}
\end{thm}
\begin{proof}The  equation (\ref{02}) can be written in the form of the following two equations:
\[
r^2(  \cos^2\varphi-\sin^2\varphi)+pr\cos\varphi+q=0
\]
and
\[
2r^2  \cos\varphi\sin\varphi+pr\sin\varphi=0
\]
or
\[
r^2 (2 \cos^2\varphi-1)+pr\cos\varphi+q=0
\]
and
\[
r\sin\varphi(2r\cos\varphi+p)=0.
\]
Except for $r=0$ or $\varphi=k\pi, k\in\{0,\overset{+}-1,\overset{+}-2,...\},$ the last equation is satisfied for $r\cos\varphi=-\frac{p}{2}.$
The previous equation is then
\[2(\frac{p}{2})^2-r^2+p(-\frac{p}{2})+q=0\]
or $r^2=q. $
\end{proof}

\begin{remark}The solution of equation (\ref{02}), that is the variable of interest $x$ and the  variable $y$, which is assumed to exist,
 jointly belong to the circle $C(\sqrt{q}).$ In other words,  the vector $(x,y)^T$ of accompanying variables takes values in                    $C(\sqrt{q})$.
\end{remark}
\section{Accompanying variables in generalized circles with respect to phs-functionals}
Let $(R^2,\oplus,\cdot)$ be the two-dimensional vector space of columns $(x,y)^T$ of real numbers where $\oplus:R^2\times R^2\rightarrow R^2$ and $\cdot:R\times R^2\rightarrow R^2$ still denote the usual componentwise vector addition and scalar multiplication, respectively. The vector $\mathfrak o=(0,0)^T$ is the additive neutral element of this space. Suppose that $||.||: R^2\rightarrow [0,\infty)$ is a \textbf{p}ositively \textbf{h}omogeneous  functional such that the set $B=\{z\in R^2:||z||\leq 1\}$ is \textbf{s}tar-shaped with respect to the inner point $\mathfrak o$. We call such functional a phs-functional.
A particular element of this class of functionals can be
 a norm or an antinorm as defined in  \cite{MoRi}. Furthermore, we call $B$  the unit disc with respect to the functional $||.||$ and its boundary
 \[S=\{z\in R^2:||z||=1\}\] the corresponding  generalized unit circle.

 The variety of accompanying variables made possible in the present work is essentially based on the introduction of numerous vector-valued vector products such as in \cite{Richter2020,Richter2022b}. The following general definition is taken from the literature.
 \begin{definition} The vector-valued  product of the  vectors $z_l=(x_l,y_l)^T, l=1,2$ from $R^2$ with respect to the phs-functional $||.||$ is defined byp
\begin{equation}\label{Mult2} z_1 \circ_{||.||} z_2=\frac{||z_1||\cdot||z_2||}{||z_1 \circledast z_2||}\cdot z_1 \circledast z_2
\end{equation} and abbreviated as $z_1\circ_{||.||} z_2=z_1\circ z_2$.
\end{definition}
Note that if the functional $||.||$ denotes the Euclidean norm $|.|_2$, then the multiplication operation $\circ$ agrees with the operation $\circledast$ and is called the circular vector multiplication or the vector-valued  vector multiplication with respect to the functional $|.|_2, \text{ thus } \circ_{|.|_2}=\circledast$.
 In addition, the   product with respect to the arbitrary  phs-functional $||.||$, $\circ$, is commutative and associative.  Further, vector addition $\oplus$ and vector multiplication $\circ$  are not distributive, in general, but a weak distributivity property in the sense of Hankel\cite{Hankel} is valid. Moreover,
 \[z_1\circ z_2=\mathfrak o
  \text{ if and only if } z_1=\mathfrak o \text{ or } z_2=\mathfrak o,
  \] \[\frac{z_1}{||z_1||}\circ\frac{z_2}{||z_2||}\in S \text{ if } z_1\neq \mathfrak o \text{ and } z_2\neq\mathfrak o
   \]
   as well as
   \[(\lambda z_1)\circ(\mu z_2)=(\lambda\mu)z_1\circ z_2 \text{ for all } z_1 \text{ and } z_2 \text{ from }R^2
   \text{ and  real numbers }\lambda \text{ and }\nu
   .\]Using the notations
   \[
z^{\circ 2}   =z\circ z
   \]for general vector-valued squaring of vector $z$, and \[Q(z)= \frac{||z||^2}{||z^{\circledast 2}||},\]
we now consider the vector-valued vector polynomial function
\[
f(z)=z^{\circ 2}+p\cdot z +q\cdot 1\hspace{-.14cm}1\]
\[
\qquad\qquad=Q(z)z^{\circledast 2}+p\cdot z +q\cdot 1\hspace{-.14cm}1
\]
and look for its vector-valued zeros that is for all $z=(x,y)^T\in R^2$ satisfying equation
\begin{equation} \label{Null3}
f(z)=\mathfrak o.
\end{equation}
If we would  introduce the notation $f=f_{\circ}$ and if $\circ$ would denote the product with respect to the Euclidean norm $|.|_2$
then we would have $f=f_\circledast$.

The following theorem follows on from the statement in Theorem \ref{Th2} and contains a first statement about where solutions of (\ref{Null3})
are located.
We assume that for the rest of this section holds
\[Q(z)\geq\frac{p^2}{4q}.\]
\begin{thm}[a]
Every solution $z=(x,y)^T$ of equation (\ref{Null3}) satisfying $y\neq 0$ belongs to the circle
\[\{(x,y)^T\in R^2:((x+\frac{q}{p})^2+y^2=(\frac{q}{p})^2\}\]
and satisfies the equations
\[x=-\frac{p}{2Q(z)}\text{\; and \;} y^2=\frac{q}{Q(z)}-(\frac{p}{2Q(z)})^2.\]
(b) If condition (\ref{nonneg}) is satisfied then every solution of (\ref{Null3}) with $y=0$ can be represented as $z=\left(
        \begin{array}{c}
          x \\
          0 \\
        \end{array}
      \right)
$ where $x$ solves (\ref{01}) according to (\ref{Lsg1}).
\end{thm}
\begin{proof}We start from the equation system
\begin{equation}
Q(z)(x^2-y^2)+px+q=0,
\end{equation}
\begin{equation}
2Q(z)xy+py=0.
\end{equation}
Let $y\neq 0$. By the second equation,
\[x=-\frac{p}{2Q(z)}.\]
With the first equation this leads to
\[y^2=\frac{q}{Q(z)}-(\frac{p}{2Q(z)})^2\]
from which finally follows the first statement. The validity of the second one can be seen as follows. Every phs-functional $||.||$ satisfies
\[Q(\left(
      \begin{array}{c}
        x \\
        0 \\
      \end{array}
    \right))=\frac{||\left(
                      \begin{array}{c}
                        x \\
                        0 \\
                      \end{array}
                    \right)
    ||^2}{||\left(
                      \begin{array}{c}
                        x^2 \\
                        0 \\
                      \end{array}
                    \right)
    ||}=1,
\] thus
\[
\left(
  \begin{array}{c}
    x \\
    0 \\
  \end{array}
\right)^{\circ 2}+p\left(
  \begin{array}{c}
    x \\
    0 \\
  \end{array}
\right)+q1\hspace{-.15cm}1=\left(
  \begin{array}{c}
    x^2+px+q \\
    0 \\
  \end{array}
\right)=\left(
          \begin{array}{c}
            0 \\
            0 \\
          \end{array}
        \right).
\]An application of Theorem 3.1(b) completes the proof
\end{proof}
We now turn over to a geometric representation of the solution of equation (\ref{Null3}). To this end, put \[N(\varphi)=||\left(
                    \begin{array}{c}
                      \cos\varphi \\
                      \sin\varphi \\
                    \end{array}
                  \right)
   || \qquad\text{ and }\qquad Q^{\circledast}(\varphi)=\frac{N^2(\varphi)}{N(2\varphi)}.\]
\begin{thm}
In terms of usual polar coordinates $r$ and $\varphi$, $\varphi \notin \{0,\pi\}$, every  solution $z=(x,y)^T$ of the vector equation (\ref{Null3}) satisfies the following two equations:
\begin{equation}\label{4.5}
\frac{N^2(\varphi)\cos^2\varphi}{N(2\varphi)}=\frac{p^2}{4q}
\end{equation}
and
\begin{equation}\label{4.6}r=\frac{\sqrt{q N(2\varphi)}}{N(\varphi)}.
\end{equation}
\end{thm}
\begin{proof}
We first remark that under polar coordinate transformation $Q(z)$ becomes
\[  \frac{||\left(
                  \begin{array}{c}
                    \cos\varphi \\
                    \sin\varphi \\
                  \end{array}
                \right)
                ||^2}{||\left(
                  \begin{array}{c}
                    \cos\varphi \\
                    \sin\varphi \\
                  \end{array}
                \right)\circledast \left(
                  \begin{array}{c}
                    \cos\varphi \\
                    \sin\varphi \\
                  \end{array}
                \right)
||}=Q^{\circledast
}(\varphi).\]
We now
write the vector equation (\ref{Null3}) as
\[Q^{\circledast}(\varphi)r^2\left(
                      \begin{array}{c}
                        \cos2\varphi \\
                        \sin2\varphi \\
                      \end{array}
                    \right)
+pr\left(
                                  \begin{array}{c}
                                    \cos\varphi \\
                        \sin\varphi \\
                                  \end{array}
                                \right)+q\left(
                                          \begin{array}{c}
                                            1 \\
                                            0 \\
                                          \end{array}
                                        \right)=\left(
                                                 \begin{array}{c}
                                                   0 \\
                                                   0 \\
                                                 \end{array}
                                               \right)
\]or as the two individual equations
\begin{equation}\label{21}
Q^{\circledast}(\varphi){r^2}(2\cos^2\varphi-1)+pr\cos\varphi+q=0
\end{equation}
and
\begin{equation}\label{22}
(2Q^{\circledast}(\varphi){r^2}\cos\varphi+pr)\sin\varphi=0.
\end{equation}
The first factor in (\ref{22}) equals zero if
\[{rQ^{\circledast}(\varphi)\cos\varphi}=-\frac{p}{2}\]
or
\begin{equation}r^2=\frac{p^2}{4\cos^2\varphi\, Q^{\circledast 2}(\varphi)}.
\end{equation}This inserted into (\ref{21}) gives
\[q={r^2}Q^{\circledast}(\varphi)=\frac{p^2}{4\cos^2\varphi\, Q^{\circledast }(\varphi)}\]or
\begin{equation}
\cos^2\varphi Q^{\circledast}(\varphi)=\frac{p^2}{4q}
\end{equation}
from where the result follows
\end{proof}
\begin{remark}If one considers the particular equation (\ref{02}) then the two conditions (\ref{4.5}) and (\ref{4.6})
reduce to (\ref{Los}).\end{remark}
\section{Hyperbolically accompanying variable}
We now assume that the real variables $x$ and $y$ occur in the range
\begin{equation}\label{HypSpace}
x\in R \text{  and  }  |y|< x
\end{equation}which can easily be restricted to a space-time model by putting $x>0$.
A hyperbolic vector-valued vector product is now defined according to \cite{Richter2022b} as
\begin{equation}\label{HypProd}
\left(
                            \begin{array}{c}
                               x_1 \\
                               y_1 \\
                             \end{array}
                           \right)
\textcircled{h}\left(
  \begin{array}{c}
    x_2 \\
    y_2 \\
  \end{array}
\right)
\left(
  \begin{array}{c}
    x_1x_2+y_1y_2 \\
    x_1y_2+x_2y_1 \\
  \end{array}
\right)
.
\end{equation}In general, such a multiplication describes a movement on a hyperbola and additionally the change between two hyperbolas. Let $C_h^+(r)=\{\left(
                              \begin{array}{c}
                                x \\
                                y \\
                              \end{array}
                            \right)\in R^2: x>0, x^2-y^2=r^2
\}$ denote a "hyperbolic half circle of radius r>0" and let $z_2$ be an arbitrary element of $C_h^+(1)$. For $z_1\in C_h^+(r_1) $ then holds $z_1 \textcircled{h} z_2\in C_h^+(r_1) $, that is $z_2$ moves $z_1$ along the "half circle"  $C_h^+(r_1)$. Such a movement is complemented by a change from the "half circle" $C_h^+(r_1)$  to the "half circle " $C_h^+(r_1r_2)$ if the element from $C_h^+(1)$ is additionally multiplied by $r_2>0$ and so that $z_2\in C_h^+(r_2), r_2>0$.
In particular, for $z=\left(
                        \begin{array}{c}
                          x \\
                          y \\
                        \end{array}
                      \right)
$, we have the equations
\begin{equation}\label{HypQuad}
z^{\textcircled{h}2}=\left(
                       \begin{array}{c}
                         x^2+y^2 \\
                         2xy \\
                       \end{array}
                     \right)
\end{equation}
and, in contrast to (\ref{I1}),
\begin{equation}\label{HypImag}
I\hspace{-.14cm}{I}^{\textcircled{h}2}=1\hspace{-.15cm}{1}.
\end{equation}
We now consider the function
\[f_{\textcircled{h}}(z)=z^{\textcircled{h}2}\oplus p\cdot z\oplus q\cdot 1\hspace{-.15cm}{1}
\]
and the equation
\begin{equation}\label{03}
f_{\textcircled{h}}(z)=\mathfrak o.
\end{equation}

\begin{thm}
If condition (\ref{pos1}) is fulfilled then the equation (\ref{03}) has no solution. If assumption (\ref{nonneg})  is satisfied then equation (\ref{03}) with $x_{1/2}$ out of (\ref{Lsg1}) has the solutions
    \[z_{1/2}=\left(
                \begin{array}{c}
                  x_{1/2} \\
                  0 \\
                \end{array}
              \right)=x_{1/2}\cdot 1\hspace{-.15cm}{1}
       \text{ and }\quad z_{3/4}=-\frac{p}{2}\cdot 1\hspace{-.15cm}{1}\overset{+}-\sqrt{\frac{p^2}{4}-q}\cdot I\hspace{-.14cm}{I}.
    \]
\end{thm}

\begin{proof}We first write the vector equation (\ref{03}) as the system of univariate equations
\begin{equation}\label{AA}
x^2+y^2+px+q=0
\end{equation}
and
\begin{equation}\label{BB}
2xy+py=0.
\end{equation}
Equation (\ref{AA}) can be rewritten as
\[
(x+\frac{p}{2})^2+y^2=\frac{p^2}{4}-q
\]
and has no solution if condition (\ref{pos1}) is satisfied. Therefore, let from now on assumption (\ref{nonneg}) be fulfilled. Equation (\ref{BB}) reads as
\[
y(2x+p)=0
\]
and can be solved if $y=0$ or $x=-\frac{p}{2}.$ \newline If $y=0$ then (\ref{AA})
coincides with (\ref{01}) and has therefore solutions $x_{1/2}$. In this case, (\ref{03}) is solved by $z_{1/2}$.
\newline
If $x=-\frac{p}{2}$ then equation (\ref{AA}) means $y^2=\frac{p^2}{4}-q$ or $y=\overset{+}-\sqrt{\frac{p^2}{4}-q}$ which leads to the  solutions $z_{3/4}$
\end{proof}

\begin{remark}
The solution behavior of a quadratic equation with accompanying variables in a hyperbolic space differs significantly from that in a Euclidean space or a space which is provided with an arbitrarily chosen phs-functional. The role a semi-antinorm in the sense of \cite{MoRi} plays in a corresponding hyperbolic space is discussed in \cite{Richter2022b}.
\end{remark}
\section{Comments}
In case condition (\ref{pos1}) is fulfilled, the well-known complex analysis literature on the solution of equation (\ref{01}) states  that it has the form
\begin{equation}\label{complex}x_{1/2}=-\frac{p}{2}\overset{+}-i\sqrt{q-\frac{p^2}{4}}
\end{equation}
and the following is said very nebulously about quantity $i$:
\begin{enumerate}
\item[P1] The quantity $i$ is not a real number.
\item[P2] The quantity $i$ satisfies the equation $i^2=-1.$
\end{enumerate}
Sometimes $i$ is not said to be just a quantity but to be a number with the properties P1 and P2. But what if the complex numbers have not yet been introduced, is a number other than a real number? For example, if $i$ ( in harmony with property P1)  denotes the "moon-number", how is the squaring of the "moon-number" defined, as in property P2 and moreover so that the same squaring-operation is also applicable to real numbers as in P2?
\newline
When graphically representing complex numbers, some authors start with a coordinate system whose axes they do not label. Figures of complex numbers shown should occasionally give then the impression that they are figures in $R^2.$ However, figures formed with imaginary numbers (or "moon-numbers") are imaginary by nature and actually therefore cannot be graphically represented. The fortunate resolution of this conflict is based on a consistent vector-space approach to complex numbers as used in \cite{Richter2020, Richter2022b} and here.
\newline
\qquad In a certain part of standard literature, $i$ is said  to be the vector
$I\hspace{-.14cm}{I}=\left(\begin{array}{c}
0 \\
1 \\
\end{array}
\right)$,
but in later steps of introducing complex numbers, for example, it is said there, in violation of mathematical correctness, that "obviously holds $1=1\hspace{-.15cm}{1}$".

Moreover, what can one imagine by the sum of the real number $-p/2$ and the nebulous expression $i\sqrt{q-p^2}$ as in (\ref{complex})? In comparison to (\ref{complex}), equation (\ref{Lsg2}) describes a  situation that is mathematically completely clarified on the basis of vector calculus.

To the best of author's knowledge, complex numbers are always introduced  in the standard literature in one way or another with such a serious lack of mathematical rigor.

The argument of unobserved contradictions when using imaginary numbers put forward as justification by some authors should from now on  be recognized as inadequacy.

If we consider the algebraic structure
\[\mathfrak C=(R^2,\oplus,\circledast,\cdot,\mathfrak o,1\hspace{-.14cm}1,I\hspace{-.19cm}I)\]
to be the complex plane
then $\mathfrak C$ is not an extension field of the real numbers because $R$ is isomorphic to a subspace of $R^2$ but not a subset. However, note that we then have
\[I\hspace{-.19cm}I^2=-1\hspace{-.14cm}1 \text{\; instead of\quad} "i^2=-1".\]

It was thanks to the genius of Gauss that he interpreted complex numbers as points in the complex number plane. At the time of writing, for example \cite{Gauss1799}, the theory of vector spaces was not yet available to him.
How cautious Gauss was towards the field of complex numbers is expressed, among other things, in his following words from \cite{Gauss1799}:
"Again, since the time when the analysts found out that there are infinitely many equations which
had no roots at all unless quantities of the form $a + b\sqrt{-1}$ were admitted as a peculiar kind of
quantities called imaginaries distinct from reals, these supposed quantities have been studied,
and have been introduced in all of analysis: with what justification, I shall not discuss here. I
shall free my proof from any help of imaginary quantities, although I might avail myself of that
liberty of which all those dealing with analysis make use."

\section{Declarations}

The author has no competing interests to declare that are relevant to the content of this article.
No funds, grants, or other support was received. The entire work  was written solely by the author.
\bibliographystyle{amsplain}

\end{document}